\documentclass[a4paper]{amsart}

\usepackage{graphics,amssymb,enumerate}
\usepackage{hyperref}
\usepackage[dvips]{graphicx}
\usepackage[all]{xy}

\def\bd{\begin{displaymath}}
\def\ed{\end{displaymath}}
\def\be{\begin{equation}}
\def\ee{\end{equation}}

\newtheorem{theorem}{Theorem}

\newtheorem{lemma}{Lemma}

\newtheorem{question}{Question}

\newtheorem{example}{Example}

\theoremstyle{definition}

\begin{document}

\title{On prime   values  of cyclotomic polynomials}

\author{Pantelis A.~Damianou}
\address{Department of Mathematics and Statistics\\
University of Cyprus\\
P.O.~Box 20537, 1678 Nicosia\\Cyprus} \email{damianou@ucy.ac.cy}

\begin{abstract}
We present several approaches  on finding  necessary and sufficient  conditions on  $n$ so that $\Phi_k(x^n)$ is irreducible where $\Phi_k$ is the $k$-th cyclotomic polynomial.
\end{abstract}

\maketitle

\section{Introduction}

\bigskip

   Let $ U_n=\{ z \in {\bf C} \ | \ z^n =1 \} $  denote the  set of $n$th roots of unity. $U_n$ is a group under complex multiplication, in fact   a cyclic group.  A complex number $\omega$ is called a {\it primitive
$n$th root of unity} provided $\omega$  is an $n$th root of unity
and has order $n$ in the group $U_n$. Such an $\omega$ generates the group, i.e.
$\omega, \omega^2, \dots, \omega^n=1 $ coincides with the set
$U_n$.  An element of $U_n$ of the form $\omega^k$ is a primitive
root of unity iff $k$ is relatively prime to $n$. Therefore the
number of primitive roots of unity  is equal to $\phi(n)$ where $\phi$ is
Euler's totient function.

 There are  three cube roots of unity: $1$, $\omega=  (-1 +\sqrt{3} i) / 2
$ and  $\omega^{2}$; only  $\omega$ and $\omega^2$ are primitive.
There are  two  primitive fourth roots of unity: $\pm i$. They
are  roots of the polynomial $x^2+1$.

We define the $n$th cyclotomic polynomial by
\begin{displaymath}
\Phi_n(x)=(x-\omega_1)(x-\omega_2) \cdots (x-\omega_s)  \ ,
\end{displaymath}
where $\omega_1, \omega_2, \ldots \omega_s$ are all the distinct
primitive $n$th roots of unity. The degree of $\Phi_n$ is of course equal to $s=\phi(n)$.
 $\Phi_n(x)$ is a monic, irreducible polynomial with integer coefficients.  The first twenty cyclotomic polynomials  are given below:

 \bd
 \begin{array}{lcl}
\Phi_1(x) & =& x-1 \\
\Phi_2(x) & =& x+1 \\
\Phi_3(x) & = &x^2+x+1 \\
\Phi_ 4(x) & =& x^2+1=\Phi_2(x^2)  \\
\Phi_ 5(x) & = &x^4+x^3+x^2+x+1 \\
\Phi_ 6(x) & =& x^2-x+1 =\Phi_3(-x)\\
\Phi_ 7(x) & =&x^6+x^5+x^4+x^3+x^2+x+1 \\
\Phi_8 (x) & =& x^4+1 =\Phi_2(x^4)=\Phi_4(x^2) \\
\Phi_9 (x) & =&  x^6+x^3+1=\Phi_3(x^3)  \\
\Phi_{10} (x) & =&x^4-x^3+x^2-x+1=\Phi_5(-x) \\
\Phi_{11} (x) & =& x^{10}+x^9+\dots +x +1 \\
\Phi_{12} (x) & = &x^4-x^2+1=\Phi_6(x^2) \\
\Phi_{13} (x) & =& x^{12}+x^{11}+\dots +x+1  \\
\Phi_{14} (x) & = &x^6-x^5+x^4-x^3+x^2-x+1=\Phi_7(-x) \\
\Phi_{15} (x) & =& x^8-x^7+x^5-x^4+x^3-x+1 =\frac{ \Phi_3(x^5) }{\Phi_3(x)} =\frac{ \Phi_5(x^3) }{\Phi_5(x)}\\
\Phi_{16} (x) & =&x^8+1=\Phi_2(x^8) =\Phi_8(x^2) \\
\Phi_{17} (x) & =& x^{16}+x^{15}+\dots +x+1  \\
\Phi_{18} (x) & = &x^6-x^3+1=\Phi_9(-x) \\
\Phi_{19} (x) & = & x^{18}+x^{17}+\dots +x+1 \\
\Phi_{20} (x) & =& x^8-x^6+x^4-x^2+1=\Phi_{10} (x^2)
\end{array}
 \ed

   A basic formula for  cyclotomic polynomials is

\begin{equation}
\label{xn} x^n -1= \prod_{d|n} \Phi_d(x) \ ,
\end{equation}
where $d$ ranges over all positive divisors of $n$. This gives a
recursive method of calculating  cyclotomic polynomials.  For
example, if $n=p$,  where $p$ is prime, then $x^p -1=\Phi_1(x) \Phi_p
(x)$ which implies that \bd \Phi_p(x)=x^{p-1}+x^{p-2}+\dots+x+1 \
.\ed

Using M\"obius inversion on (\ref{xn}) we obtain an  explicit formula for cyclotomic polynomials  in terms
of the M\"obius function:

\be \label{mobius} \Phi_n(x)=\prod_{d|n} (x^d-1) ^{ \mu (
\frac{n}{d} ) } \ . \ee

We record the following formulas where $p$ is a prime and $n,m$ are positive integers. The formulas can be found in many elementary  textbooks.

 If $p |m$ then
\be \label{eq3} \Phi_{pm}(x)=\Phi_m (x^p)  \ee

If $p\not|m$
\be \label{eq4} \Phi_{pm}(x) \Phi_m(x)=\Phi_m(x^p)  \ee

\be \label{eq5} \Phi_{p^k}(x)=\Phi_p \left( x^{p^{k-1}} \right)  \ee

If $n$ is odd, $n>1$
\be \label{eq6}  \Phi_{2n}(x)=\Phi_n (-x) \ee

Let $n=p_1^{\alpha_1} \dots p_s^{\alpha_s}$ and define $m=p_1 p_2 \dots p_s$. Then
\be \label{eq7}  \Phi_n(x)=\Phi_m\left( x^{ \frac{n}{m}} \right) \ee

In this paper we deal with the following question:
\smallskip

\begin{question}  Find conditions on $k$ and $n$ so that
$\Phi_k(x^n)$ is irreducible. This gives a necessary condition for
$\Phi_k(a^n)$ to be a prime if $a,n,k \in {\bf Z}^{+}$, $k>1$, $a>1$.
\end{question}

 A more
difficult question which we do not address here is the following.
For fixed $k$ and $n$ (or fixed $k$ and $a$) are there infinitely many primes of the form
$\Phi_k(a^n)$?  e.g. If $n=1$,  are there infinitely many primes of the form $\Phi_k(a) \ a=1,2, 3, \dots$ where $\Phi_k$ is a given cyclotomic polynomial?  Another possibility is to fix  $a$ and $n$.  For example primes of the form  $\Phi_k(1)$, $\Phi_k(2)$ are investigated  in  \cite{gallot}.

\begin{example} [k=4]
 Let us consider the case  $k=4$, i.e.
$\Phi_4(x)=x^2+1$.  We can  rephrase the question as follows: When
is $x^{2n}+1$ irreducible?  It is elementary to see that $n$
 must  be of the form $2^j$ for some $j\ge 0$. Therefore the
polynomial should be of the form \bd x^{2^n} +1 \ . \ed
 In order for $a^{2^n}+1$ to be prime, it is easy to see that $a$
 should be even. In the case $a=2$,  we end-up with the Fermat primes  of
 the form
 \bd F_n=2^{2^n}+1 \ . \ed
 The only known Fermat primes are $F_0$, $F_1$, $F_2$, $F_3$ and $F_4$.
\end{example}

\begin{example}[k=1,2]
For completeness let us consider also the cases $k=1$ and $k=2$.
In the case $k=1$ we have $\Phi_1 (x^n) =x^n -1$ which is always
reducible, unless $n=1$.  Numbers of the
form $a-1$ are prime for infinitely many values of $a$ as was
observed by Euclid. In connection with Question 1 we make the following remark: If  $\Phi_1(a^n)=a^n-1$ is prime then $n=p$ for some prime $p$ and $a=2$.  Therefore for Mersenne primes of the form $2^{p}-1$ it turns out  that the corresponding polynomial $x^p-1$  is  reducible. But this is an exceptional case (That is why we assume $k>1$ in Question 1). The only other exception is when $a=1$ and $k$ has at least two distinct prime divisors.

In the case $k=2$ we have $\Phi_2 (x^n)=x^n+1$. If $n$ is odd,
then $\Phi_2(x^n) $ is divisible by $x+1$. Therefore $n$ must be
even and we end-up with a polynomial of the form $x^{2n}+1$, a
case that has been  already considered.
\end{example}

Most of the results of this paper are well-known. But to our knowledge they are scattered in various
places and the problem in question has not been addressed before. In the next two sections we consider in detail the case $k=3$.   Although the result is stated and proved in two different ways   for this particular case, the method of proof is applicable  for other values of $k$. In section 4  we give a third proof which covers the general case of arbitrary $k$. In section 5 we outline two  other proofs of  theorem \ref{gen}, one due to S. Golomb,  and one due to L. Leroux.

\section{The case $k=3$}

In this section we use the notation \bd p_n(x)=\Phi_3(x^n)=x^{2n}+x^n+1 \ . \ed
It is easy to see that $p_n(x)$ is a product of cyclotomic
polynomials.  We  note that $p_n(x)$ is  a monic polynomial,  with
integer coefficients,  having all of its roots in the
unit disc (in fact on the unit circle).  It follows from a
classical result of Kronecker, see \cite{damianou},   that $p_n(x)$
is  a product of cyclotomic polynomials.  Showing that the roots
of this polynomial are roots of unity is easy. Let $\rho$ be such
a root, i.e. $\rho^{2n}+\rho^n +1=0$.  We let $\lambda=\rho^n$ and
thus $\lambda^2+\lambda+1=0$. In other words $\lambda$ is a
primitive cubic root of unity. Therefore $\rho^{3n}=1$, i.e.,
$\rho$ is a $3n$ root of unity. This fact follows even easier from
the identity \bd x^{3n}-1=(x^n-1)(x^{2n}+x^n+1) \ . \ed

In this section we will prove the following result:

\begin{theorem}\label{k3}
If $a^{2n}+a^n +1$ is prime then $n=3^{j}$ for some $j=0,1,2, \dots $.
\end{theorem}

Actually, we will show that $x^{2n}+x^n+1$ is irreducible iff $n=3^{j}$ for some $j=0,1,2, \dots $.

Let us see some examples for small values of $n$
\begin{example}
\bigskip
\noindent

\begin{itemize}
\item $n=1$.
\bd p_1(x)=x^2+x+1=\Phi_3(x) \ed
which is irreducible.
\item
$n=2$
\bd p_2(x)=x^4+x^2+1=(x^2+x+1)(x^2-x+1)=\Phi_3(x) \Phi_6(x) \ . \ed

\item $n=3$ \bd p_3(x)=x^6+x^3+1=\Phi_9(x)  \ed which is
irreducible. \item $n=4$ \bd p_4(x)=x^8+x^4+1=\Phi_3(x) \Phi_6(x)
\Phi_{12}(x) \ . \ed
\item $n=5$

\bd p_5(x)=x^{10}+x^5+1= \Phi_3(x) \Phi_{15}(x) \ . \ed

\item $n=6$

\bd p_6(x)=x^{12}+x^6+1= \Phi_9(x) \Phi_{18}(x) \ . \ed

\end{itemize}
\end{example}

One may wonder about the cyclotomic polynomials that appear in the
factorization of $p_n(x)$.   We will shortly give a formula that
precisely determines the cyclotomic factors of $x^{2n}+x^n+1$.

Let us examine briefly the behavior of $x^{2n}+x^n +1$ for fixed
$n$.
 There is strong experimental evidence that $a^2+a+1$ is
prime for infinitely many values of $a$, i.e. \bd a=1,2,3,5,
6,8,12,14,15,17,20,21,24,27,33,38,41,50,54, 57, 59, 62, 66, 69, 71, \ed
\noindent $ 75,77,78,80, 89,90,99,
\dots  \ . $ There are a  total of 189 such values under 1000.

Similarly, we note  that $a^6+a^3+1$ is likely to be  prime for
infinitely many values of $a$, i.e.

 \bd  a=1,2,3,8,11,20,21,26,30,50, 51,56,60,78,98, \dots  \ . \ed There are a
total of 79 such values under 1000.

Finally, $a^{18}+a^9+1$ is prime  for \bd  a=1,2,11,44,45,56,62,63,110,170,219,234,245,261,263,  \dots \ed  There are a total of 47 values under 1000.

On the other hand, we may fix $x$ and investigate  the behavior of this
expression as a function of $n$. This will be the analogue  of
Fermat primes. For example for $x=2$ we may ask whether \bd 2^{2
3^n}+2^{3^n}+1 \ed is prime for infinitely many values of $n$.  It
is certainly true for $n=0,1,2$. Similarly we may ask whether \bd
3^{2 3^n}+3^{3^n}+1 \ed is prime for certain values of $n$.  It is
true for $n=0, 1$.

We now prove the following lemma, see \cite{nieto}:
\begin{lemma} \label{lemma1}

Let  $f,g \in {\bf Z}[x]$ and denote by $c(f)$ the content of $f$.  If $c(g)|c(f)$ and $g(n)|f(n)$ for infinitely many integers $n$, then $g|f$ in  ${\bf Z}[x]$.

\end{lemma}
\begin{proof}
Divide $f$ by  $g$ in  $\mathbf{Q}[x]$ to obtain
\bd f(x)=g(x) q(x)+r(x)   \ed
with $q(x)$, $r(x)$ in  $\mathbf{Q}[x]$  and deg $r(x) < $ deg $g(x)$.

Let $m$ be the least common multiple of the denominators of $q$ and $r$ so that $mq$ and $mr$ are in $\mathbf{Z}[x]$.
Let $n_i$ be an infinite sequence such that $g(n_i)| f(n_i)$ for all $i$.   Then
\bd m  \frac{  f(n_i) } {g(n_i)}  - m q(n_i) = m \frac{ r(n_i)}{ g(n_i)}  \  . \ed
Note that the left hand side is always an integer. Therefore the right hand side is also an integer for all $i$.   But since the degree of $r$ is strictly less than the degree of $g$  we have
\bd \lim_{i \to \infty} \frac{ m r(n_i)} {g(n_i)} =0  \ . \ed
This implies that $r(n_i)=0$ for big $i$.  As a result $r(x)$ must be identically zero.
We conclude that
\bd f(x) = g(x) q(x)  \ed
with $q(x) \in \mathbf{Q}[x]$.   Now we use Gauss Lemma:
\bd m c(f)=c(g) c(mq)  \ ,  \ed
\bd m \frac{c(f)}{c(g)} =c (mq) \ . \ed
The left hand side is an integer by assumption and a multiple of $m$. This shows that all the coefficients of $m q(x)$ are a multiples of $m$. Therefore, $q(x) \in \mathbf{Z}[x]$.
\end{proof}

The following lemma will be used in the proof of Theorem (\ref{k3}).

\begin{lemma} \label{lemma2}
\bd \Phi_k(x) | \Phi_k(x^n)  \ed
iff  $(n,k)=1$.
\end{lemma}

\begin{proof}

We  prove the following more general result:  Let $k\ge 2$. The only polynomials with coefficients in ${\bf Z}$, monic and irreducible which satisfy  $p(n)| p(n^k)$ for infinite values of $n$ are
\begin{enumerate}
\item $1$
\item $x$
\item $\Phi_j$ with $(j,k)=1$.
\end{enumerate}

Obviously, $1$ and $x$ satisfy the conditions of the Theorem and
we consider them as  trivial solutions. If $p$ is any monic
polynomial satisfying $p(n)| p(n^k)$  for infinite values of $n$
then \bd p(x^k)=p(x) Q(x) \ed for some polynomial $Q(x) \in {\bf
Z}[x] $ (lemma \ref{lemma1}). Let $\zeta$ be a root of $p$.  It follows that $p(
\zeta^k)=p(\zeta)Q(\zeta)=0$ and therefore $\zeta^k$ is also a
root. Repeating the argument we see that \bd \zeta,  \zeta^k,
\zeta^{k^2}, \dots    \ed are all roots of $p(x)$. But $p(x)$ has
only a finite number of roots and therefore two powers of $\zeta$
should be equal. Since $\zeta \not=0$ (otherwise $p$   would be
reducible)  it follows that $\zeta $ is a root of unity. If
$\zeta$ is primitive of order $j$ then all the primitive roots of
order $j$ are also roots of $p$. Therefore $p$ is a multiple of
$\Phi_j$ and being irreducible forces  $p=\Phi_j$.  In order for
$\zeta^k$ to be primitive of order $j$ we must have $(k,j)=1$.
\end{proof}

As an immediate application of the Lemma we note that
\bd \Phi_3(x) | \Phi_3(x^n) \ed  in case $(3,n)=1$.   Therefore the polynomial $p_n(x)$ is reducible in case $n$ is relatively prime to $3$.

As another application of this result we mention the following "millennial Polynomial problem" \cite{caragea}:
Find all monic, irreducible polynomials of degree $2000$ with integer coefficients such that
$p(n)| p(n^2)$ for every natural number $n$. The only solutions are $\Phi_{j}$,  where $j=2525,3333,3765,4125$.

We now proceed to the proof of  theorem (\ref{k3}).

\begin{proof}
Suppose $n=3^{j-1}m$  where $(3,m)=1$ with $m>1$, $j>1$.  First note that
\bd \Phi_{3^j}(x)=\Phi_3 \left( x^{3^{j-1}} \right)=x^{ 2 3^{j-1} }+x^{3^{j-1}}+1  \ . \ed
Therefore
\bd \Phi_{3^j}(x^m)=x^{ 2 m 3^{j-1} }+x^{m 3^{j-1}}+1 =\Phi_3(x^n) \ . \ed

Since $(3^j, m)=1$, it follows that
\bd \Phi_{3^j}(x^m)=\Phi_3 (x^n) \ed
is reducible by lemma \ref{lemma2}.

If $n=3^j$ then
\bd p_n(x)=x^{ 2 3^{j} }+x^{3^{j}}+1=\Phi_{3^{j+1}}(x) \ed
which is irreducible.
\end{proof}

\section{Another Approach}
In this section we give a second proof of theorem 1 and in addition we determine the factorization of $p_n(x)$.
We are still in the case $k=3$.
We have \bd x^{3n}-1=\prod_{d|3n} \Phi_d (x) \ , \ed and

\bd  x^{n}-1=\prod_{j|n} \Phi_j (x) \ . \ed

Since \bd x^{3n}-1=(x^n-1)(x^{2n}+x^n+1) =(x^n-1) p_n(x) \ed

we obtain

\be \label{div3}  p_n(x)=\frac{  \prod_{d|3n} \Phi_d (x) }
{\prod_{j|n} \Phi_j(x)} \ . \ee

This formula can be used to determine the factorization of $p_n$
for specific values of $n$.

\begin{example}\label{ex4}   Take $n=12$.

\bd  p_{12}(x)=\frac{  \prod_{d|36} \Phi_d (x) } {\prod_{j|12}
\Phi_j(x)}=\Phi_9(x) \Phi_{18}(x) \Phi_{36}(x)  \ . \ed

\end{example}

We deduce from formula (\ref{div3}) that the polynomial $p_n(x)$ is
reducible only if the difference between the number of divisors of
$3n$ from the number of divisors of $n$ is strictly bigger than
$1$.  If we denote by $\tau(n)$ the number of positive divisors of $n$ then the condition is simply
 $\tau(3n)-\tau(n) >1$.

We can now give a second proof of theorem (\ref{k3}).

\begin{proof}
Suppose $n=3^{j-1}m$  where $(3,m)=1$ with $m>1$, $j>1$. Then \bd
\tau(3n)-\tau(n)=\tau(3^j m)-\tau(3^{j-1}m)=(j+1)\tau(m)-j
\tau(m)=\tau(m)>1 \ . \ed
\end{proof}

Note that this line of proof generalizes easily for the case where $k=p$ a prime number.
In such a case,

\bd \Phi_p(x)=x^{p-1}+x^{p-2}+ \dots +x+1  \ . \ed

\bd p_n(x)=\Phi_p(x^n)=x^{n(p-1)}+x^{n(p-2)}+ \dots +x^n+1  \ .
\ed

As before we have

\bd x^{pn}- 1=(x^n-1) p_n(x)  \ . \ed

Finally

\be \label{divp}  p_n(x)=\frac{  \prod_{d|pn} \Phi_d (x) }
{\prod_{j|n} \Phi_j(x)} \ . \ee

Suppose $n=p^{j-1}m$  where $(p,m)=1$ with $m>1$, $j>1$. Then \bd
\tau(pn)-\tau(n)=\tau(p^j m)-\tau(p^{j-1}m)=(j+1)\tau(m)-j
\tau(m)=\tau(m)>1 \ed

The conclusion: $p_n(x)$ is irreducible iff $n=p^j$ for $j=0,1,2,
\dots$.

\section{The general case}

We now consider the most general case. We find  conditions on $k$
and $n$ so that $\Phi_k(x^n)$ is irreducible. Suppose
$k=p_1^{\alpha_1}p_2^{\alpha_2} \dots p_s^{\alpha_s}$, with $\alpha_i >0$.  We claim
the following:
\begin{theorem} \label{gen}
$\Phi_k(x^n)$ is irreducible iff $n=p_1^{\beta_1}p_2^{\beta_2}
\dots p_s^{\beta_s}$ for some $\beta_j \ge 0$.
\end{theorem}

\begin{proof}

The proof is based on equation (\ref{eq4})  which we noted in the introduction:  If $p$ does not divide $k$  then

\bd \Phi_{pk}(x) \Phi_k(x)=\Phi_k(x^p)  \ . \ed

For completeness we give a proof of this property.

\bd \Phi_{pk}(x)=\prod_{ d|pk \ p|d} (x^d-1)^{ \mu \left( \frac{pk}{d}\right)}
\prod_{ d|pk \ p\not|d} (x^d-1)^{ \mu \left( \frac{pk}{d} \right)} = \ed

\bd \Phi_k(x^p)\prod_{ d|pk \ p\not|d} (x^d-1)^{ \mu \left( \frac{pk}{d}\right) }
\ . \ed But $\mu \left( \frac{pk}{d} \right)=\mu(p) \mu ( \frac{k}{d} )=-\mu (
\frac{k}{d} )$ and  the result follows.

Suppose $p$ is a prime divisor of $n$ such that $p$ does not
divide $k$. Then $n=p t$ for some integer $t$.

\bd \Phi_k(x^n)=\Phi_k(x^{pt})=\Phi_k( (x^t)^p)=\Phi_{pk} (x^t)
\Phi_k(x^t) \  . \ed
This implies that $\Phi_k(x^n)$ is reducible.
On the other hand if $n=p_1^{\beta_1}p_2^{\beta_2} \dots
p_s^{\beta_s}$ then

\be \label{eqx} \Phi_k(x^n)=\Phi_k \left(x^{p_1^{\beta_1}p_2^{\beta_2} \dots
p_s^{\beta_s} } \right)  =\Phi_k \left(  x^{ \frac{\lambda}{k} }
\right)=\Phi_{\lambda}(x)  \ , \ee
 where $\lambda=\prod_{i=1}^s
p_i^{\alpha_i+\beta_i}=kn $. Therefore, in this case $\Phi_k(x^n)$ is irreducible.
\end{proof}

Note that if we set $n=p$,  where  $p$ is  prime, in equation (\ref{eqx})  we obtain (\ref{eq3}). 

\begin{example} Take $k=6$.  We have $\Phi_6(x)=x^2-x+1$.  As
usual we denote by $p_n(x)=\Phi_6(x^n)=x^{2n}-x^n+1$.  According
to theorem (\ref{gen}),  $p_n(x)$ is irreducible iff $n=2^{\alpha} 3^{\beta}$
for some $\alpha$, $\beta \ge 0$. Indeed

\begin{itemize}
\item $n=1$. \bd p_1(x)=x^2-x+1=\Phi_6(x) \ed  \item $n=2$ \bd
p_2(x)=x^4-x^2+1=\Phi_6(x^2) =\Phi_{6.2}(x)=\Phi_{12}(x) \ed

\item $n=3$ \bd p_3(x)=x^6-x^3+1=\Phi_6(x^3)=\Phi_{18}(x)  \ . \ed

\item $n=4$ \bd p_4(x)=x^8-x^4+1=\Phi_6(x^4)=\Phi_6(
x^{\frac{24}{6} } )=\Phi_{24}(x) \ed  \item $n=5$ \bd
p_5(x)=x^{10}-x^5+1= \Phi_6(x) \Phi_{30}(x) \ . \ed \item $n=6$

\bd p_6(x)=x^{12}-x^6+1= \Phi_6(x^{ \frac{36}{6}})=\Phi_{36}(x)  \
. \ed

\item $n=7$
\bd p_7(x)=x^{14}-x^7+1=\Phi_{6}(x) \Phi_{42}(x) \
. \ed

\end{itemize}
\end{example}

Note that in this example the factorization is determined by the formula:
\bd p_n(x)= \frac {\prod_{d|6n} \Phi_d(x)  \prod_{s|n} \Phi_s(x) } { \prod_{j|2n} \Phi_j(x) \prod_{t|3n} \Phi_t(x) } \ . \ed

\section{Related works}

After the posting of  preprint  \cite{damianou2}  we became aware of  the following two alternative  approaches to  theorem (\ref{gen}).  The first is  due to Solomon  Golomb dating back to 1978.  The proof in \cite{solomon} is the following:

Since the roots of $\Phi_k(x)$ are the primitive  $k$th roots of unity, the roots of $\Phi_k(x^n)$ are the $n$th roots of these, which are the $(nk)$th roots of unity, including all the primitive $(nk)$th roots of unity. Thus $\Phi_{nk}(x)$ divides $\Phi_k(x^n)$, and $\Phi_k(x^n)$ can only be irreducible if it is equal to $\Phi_{nk}(x)$. On the other hand, since all cyclotomic polynomials are irreducible over the rationals, $\Phi_k(x^n)$ is irreducible if and only if it equals $\Phi_{nk}(x)$. Since $\Phi_{nk}(x)$ divides $\Phi_k(x^n)$, these two polynomials are equal if and only if they have the same degree.  Therefore we must have $\phi(kn) =k \phi(n)$.  This condition is equivalent to
\bd \prod_{p |nk} \left( 1 - \frac{1}{p} \right) =\prod_{p |k} \left( 1 - \frac{1}{p} \right)  \ . \ed
Clearly this condition holds iff $n$ has no  prime factors not contained in $k$.

There is also a proof along the same lines   by  Louis Leroux in \cite{leroux} which also includes a nice  formula for factoring $\Phi_k(x^n)$.
Recall that if $n=\prod_i p_i^{\alpha_i}$ then $\alpha_i={\rm ord}_{p_i}(n)$.
Let
\bd m=\prod_{p|k} p^{ {\rm ord}_p \, (n)}  \ed
and $N=\frac{n}{m}$.   Then
\be \label{fact} \Phi_k(x^n)=\Phi_{km} (x^{N})= \prod_{d|N} \Phi_{k m d} (x)  \ . \ee
To prove the first equality it is enough to remark that the two polynomials have the same roots. Since these polynomials  are monic, square free and of the same degree they must be equal.  Similarly, each root of the first polynomial is a root of the third. These two polynomials also have the same degree due to the fact
\bd \sum_{d|N} \phi(d)=N   \ . \ed
If each prime in $n$ is also in $k$ then $m=n$ and $N=1$.  Therefore using (\ref{fact}) we conclude that
$\Phi_k(x^n)=\Phi_{kn}(x)$ which is irreducible. Also note that the number of factors of $\Phi_k(x^n)$ is $\tau(N)$.
Let us repeat example (\ref{ex4}) using formula (\ref{fact}).
\begin{example} \label{ex5}
We have $k=3$ and $n=12$.  The order of $3$ in $12$ is $1$ and therefore $m=3$, $N=4$.   We obtain
\bd \Phi_3(x^{12})=\prod_{d|4} \Phi_{9d}(x)=\Phi_9(x) \Phi_{18}(x) \Phi_{36}(x) \ . \ed
\end{example}

\end{document}